\documentclass{amsart}
\usepackage[utf8]{inputenc}

\usepackage{amssymb,comment}
\usepackage{mathtools}
\usepackage{mathrsfs}

\usepackage{tikz-cd}

\usepackage{csquotes}

\usepackage{xcolor}
\definecolor{darkgreen}{rgb}{0,0.5,0}
\definecolor{darkblue}{rgb}{0,0,0.8}
\definecolor{darkred}{rgb}{0.8,0,0}

\usepackage[pdfencoding=auto,colorlinks,citecolor=darkgreen,linkcolor=darkblue,urlcolor=darkred]{hyperref}

\usepackage[capitalize]{cleveref}
\crefname{definition}{Definition}{Definitions}
\crefname{lemma}{Lemma}{Lemmata}
\crefname{proposition}{Proposition}{Propositions}
\crefname{theorem}{Theorem}{Theorems}
\crefname{corollary}{Corollary}{Corollaries}
\crefname{conjecture}{Conjecture}{Conjectures}
\crefname{algorithm}{Algorithm}{Algorithms}
\crefname{example}{Example}{Examples}
\crefname{remark}{Remark}{Remarks}
\crefname{question}{Question}{Questions}
\crefname{enumi}{Step}{Steps}
\crefname{assumption}{Assumption}{Assumptions}

\usepackage[english]{babel}
\usepackage{url}

\usepackage{enumitem}
\usepackage{colonequals}

\DeclareFontFamily{U}{wncy}{}
\DeclareFontShape{U}{wncy}{m}{n}{<->wncyr10}{}
\DeclareSymbolFont{mcy}{U}{wncy}{m}{n}
\DeclareMathSymbol{\Sha}{\mathord}{mcy}{"58}

\usepackage{microtype}

\theoremstyle{plain}
\newtheorem{theorem}{Theorem}[section]
\newtheorem{lemma}[theorem]{Lemma}

\newtheorem{proposition}[theorem]{Proposition}

\newtheorem{question}[theorem]{Question}
\newtheorem{example}[theorem]{Example}
\newtheorem{remark}[theorem]{Remark}
\newtheorem{notation}[theorem]{Notation}

\theoremstyle{definition}

\theoremstyle{remark}

\def\ol#1{\overline{#1}}
\def\wt#1{\widetilde{#1}}

\def\Alphabet{A,B,C,D,E,F,G,H,I,J,K,L,M,N,O,P,Q,R,S,T,U,V,W,X,Y,Z}
\def\alphabet{a,b,c,d,e,f,g,h,i,j,k,l,m,n,o,p,q,r,s,t,u,v,w,x,y,z}
\def\endpiece{xxx}
\def\makeAlphabet[#1]{\expandafter\makeA#1,xxx,}
\def\makealphabet[#1]{\expandafter\makea#1,xxx,}
\def\makeA#1,{\def\temp{#1}\ifx\temp\endpiece\else%
	\mkbb{#1}\mkfrak{#1}\mkbf{#1}\mkcal{#1}\mkscr{#1}\mkbs{#1}\expandafter\makeA\fi}%
\def\makea#1,{\def\temp{#1}\ifx\temp\endpiece\else\mkfrak{#1}\mkbf{#1}\mkbs{#1}\expandafter\makea\fi}%
\def\mkbb#1{\expandafter\def\csname bb#1\endcsname{\mathbb{#1}}}
\def\mkfrak#1{\expandafter\def\csname fr#1\endcsname{\mathfrak{#1}}}
\def\mkbf#1{\expandafter\def\csname b#1\endcsname{\mathbf{#1}}}
\def\mkcal#1{\expandafter\def\csname c#1\endcsname{\mathcal{#1}}}
\def\mkscr#1{\expandafter\def\csname s#1\endcsname{\mathscr{#1}}}
\def\mkbs#1{\expandafter\def\csname bs#1\endcsname{{\boldsymbol{#1}}}}
\def\makeop[#1]{\xmakeop#1,xxx,}
\def\mkop#1{\expandafter\def\csname #1\endcsname{{\mathrm{#1}}}} %
\def\xmakeop#1,{\def\temp{#1}\ifx\temp\endpiece\else\mkop{#1}\expandafter\xmakeop\fi}%
\def\makeup[#1]{\xmakeup#1,xxx,}
\def\mkup#1{\expandafter\def\csname #1\endcsname{{\mathrm{#1}\,}}} %
\def\xmakeup#1,{\def\temp{#1}\ifx\temp\endpiece\else\mkup{#1}\expandafter\xmakeup\fi}%
\makeAlphabet[\Alphabet]
\makealphabet[\alphabet]
\makeop[Hom,Tor,Ext,holim,hocolim,dim,Sel,GL,SL,H,ord,Sym,Gal,Ann,ind,Aut,End,cd,Frob,Gr,sp,Tot,sm,an,Pic,NS,Br,tors,Reg,new,cts,ur,alg,pd,disc,cyc,rk,cork,ab,nr,tors,Iw,res,inf,Cl,Quot]
\makeup[Spec,Proj,Spwf,Sp,Sh,Spf,Sch,id,dR,rig,cris,im,coker]

\newcommand{\C}{\mathbf{C}}
\newcommand{\F}{\mathbf{F}}

\newcommand{\Q}{\mathbf{Q}}
\newcommand{\Qbar}{\ol\Q}
\newcommand{\Z}{\mathbf{Z}}

\newcommand{\eps}{\varepsilon}
\renewcommand{\epsilon}{\varepsilon}
\renewcommand{\theta}{\vartheta}
\renewcommand{\phi}{\varphi}

\newcommand{\mathup}[1]{\text{\textup{#1}}}
\renewcommand{\H} {\ensuremath{\mathup{H}}}

\newcommand{\Het}{\H_\et}

\newcommand{\defeq}{\colonequals}

\newcommand{\inj}{\hookrightarrow}

\numberwithin{equation}{section}

\newcommand{\et}{\mathrm{\acute{e}t}}
\newcommand{\Tr}{\operatorname{Tr}}

\begin{document}

\title[Point counts of abelian varieties and their zeta function]{Point counts of abelian varieties\\ over finite fields\\ determining their zeta function}

\author{Shiva Chidambaram}
\address{Department of Mathematics, University of Wisconsin-Madison}
\email{chidambaram3@wisc.edu}

\author{Timo Keller}
\address{Leibniz Universität Hannover, Institut für Algebra, Zahlentheorie und Diskrete Mathematik, Welfengarten 1, 30167 Hannover, Germany}
\email{math@kellertimo.de}
\urladdr{\url{https://www.timo-keller.de}}


\date{\today}

\begin{abstract}
Let $A$ be an abelian variety of dimension $g$ over a finite field $\F_q$. We show
that if $q$ is sufficiently large relative to $g$, the $g$ point counts
$\#A(\F_{q^i})$ for $1 \leq i \leq g$ determine the zeta function of $A$,
equivalently the characteristic polynomial of its Frobenius endomorphism, and
hence the isogeny class of $A$. This count is best possible for $g=2$ and $g=4$,
but not in general: for $g=3$ two point counts already determine the zeta
function, whereas a single count never does. The proof combines the functional equation of the
$L$\nobreakdash-polynomial with Newton's identities and an inductive error analysis that
controls the power sums of the inverse Frobenius eigenvalues with enough
precision to recover them, as integers, by rounding.
\end{abstract}

\maketitle

\section{Introduction}

Let $A$ be an abelian variety of dimension $g$ over the finite field $\F_q$.
Its zeta function is the rational function
\[
    Z(A, T) = \frac{P_1(T)P_3(T)\cdots P_{2g-1}(T)}{P_0(T)P_2(T)\cdots P_{2g}(T)},
    \qquad P_i(T) = \det\bigl(1 - \Frob_q T \mid \Het^i(A_{\ol\F_q}, \Q_\ell)\bigr),
\]
and it is determined by the \emph{$L$\nobreakdash-polynomial} $f_A(T) = P_1(T)$, the reverse
characteristic polynomial of the Frobenius endomorphism, since the isomorphisms
$\Het^i \cong \wedge^i \Het^1$ give $P_i = \wedge^i P_1$. Here, for a polynomial
$h(T) = \prod_{j=1}^{d}(1 - \gamma_j T)$, we write $\wedge^i h$ for the polynomial
\[
    (\wedge^i h)(T) = \prod_{1 \leq j_1 < \cdots < j_i \leq d}\bigl(1 - \gamma_{j_1}\cdots\gamma_{j_i}\,T\bigr)
\]
of degree $\binom{d}{i}$, whose reciprocal roots are the products of $i$ distinct
reciprocal roots $\gamma_j$ of $h$; in particular $\deg P_i = \binom{2g}{i}$.
The $L$\nobreakdash-polynomial in turn determines, and is determined by, the isogeny class
of $A$ by Tate's theorem. On the other hand, $Z(A,T)$ encodes all the point
counts $\#A(\F_{q^n})$, $n \geq 1$. It is therefore natural to ask how many of
these point counts are needed in order to recover $Z(A, T)$ itself.

Since $f_A$ has $2g$ free coefficients, one expects the $2g$ point counts
$\#A(\F_{q^n})$, $1 \leq n \leq 2g$, to suffice; Kedlaya \cite[\S8]{Kedlaya2006} showed that this holds once $q > 256\,g^2$,
since then $\#A(\F_{q^{2n}})$ is available for every $1 \leq n \leq g$ and the
truncated M\"obius inversion of \cref{lem:small-i} recovers $f_A$. The functional
equation satisfied by $f_A$ halves the number of free coefficients to $g$, which
suggests that already $g$ point counts should suffice. Our main result confirms
this for $q$ large.

\begin{theorem}\label[theorem]{thm:main}
There is an explicit constant $Q(g)$, depending only on $g$, with the following
property. If $A$ is an abelian variety of dimension $g$ over $\F_q$ with
$q > Q(g)$, then the $g$ point counts
\[
    \#A(\F_{q^i}), \qquad 1 \leq i \leq g,
\]
determine the $L$\nobreakdash-polynomial $f_A$, and hence the zeta function $Z(A, T)$ and
the isogeny class of $A$. One may take $Q(g) = \bigl(16 g^3 p(2g)\bigr)^{2g+2}$,
where $p(\,\cdot\,)$ denotes the partition function.
\end{theorem}

For $g = 1$ the statement is classical and holds for every $q$: the single point
count $\#A(\F_q) = q + 1 + a_1$ determines the coefficient $a_1$, hence
$f_A = 1 + a_1 T + q T^2$, directly. We therefore assume $g \geq 2$ throughout.

It is instructive to phrase the problem geometrically. It is cleanest to use the
\emph{trace polynomial} $h(T) = \prod_{i=1}^{g}(T - \beta_i)$ with $\beta_i =
\alpha_i + q/\alpha_i$ (introduced in \cref{sec:remarks}) in place of $f_A$ itself: its
$g$ coefficients are natural coordinates on an affine space $\mathbf{A}^g$, and the
functional equation is built in. Each point count is a \emph{cyclic resultant} of
$f_A$, since
\[
    c_n = \#A(\F_{q^n}) = \prod_{i=1}^{2g}(1 - \alpha_i^n) = \prod_{\zeta^n = 1} f_A(\zeta)
        = \pm\operatorname{Res}(f_A,\, T^n - 1).
\]
Prescribing the first $g$ point counts therefore cuts out a subscheme of
$\mathbf{A}^g$ by $g$ equations. Passing by M\"obius inversion from the $c_n$ to
the cyclotomic norms $F_n = \prod_{\operatorname{ord}\zeta = n} f_A(\zeta) =
\operatorname{Res}(f_A, \Phi_n)$, where $\Phi_n$ is the $n$-th cyclotomic
polynomial, the $n$-th equation has degree $\phi(n)$; for a complete intersection
B\'ezout's theorem then predicts a zero-dimensional fibre of degree
$\prod_{i=1}^{g}\phi(i)$ over $\C$. Our theorem is the arithmetic statement that,
however many complex points such a fibre carries, \emph{at most one} of them is an
integral point lying in the Weil region cut out by the Riemann hypothesis.

Recovering a polynomial from its cyclic resultants is a natural inverse problem in
its own right, and has applications in the theory of dynamical zeta functions, knot theory (Alexander polynomials and Lehmer's problem), etc. Hillar \cite{Hillar2005} showed that the sequence of cyclic resultants of a generic monic polynomial $P$ over an algebraically closed field of characteristic $0$, determines the polynomial uniquely. Hillar and Levine \cite{HillarLevine2007} made this effective by showing that the first $2^{d+1}$ cyclic resultants determine $P$, where $d = \deg(P)$. They also showed that if $P$ is a generic monic reciprocal polynomial, the first $2 \cdot 3^{d/2}$ cyclic resultants suffice.
This is still far from the conjecture of Sturmfels and Zworski \cite[Conjecture 1.2]{HillarLevine2007}, which predicts that the first $d/2+1$ cyclic resultants should determine a generic monic reciprocal polynomial.

Our $f_A$ is of degree $2g$ and satisfies the ``$q$-reciprocal'' functional equation
$f_A(T) = T^{2g}q^{-g}f_A(q/T)$, which halves the
number of unknown coefficients to $g$. \cref{thm:main} does
better than Sturmfels--Zworski's prediction for reciprocal polynomials, by using extra arithmetic constraints. Concretely, the gain
comes from trading the generality of the reconstruction problem over an algebraically closed field, for the arithmetic constraints special to our setting: integrality of the coefficients, Riemann hypothesis, and the functional equation.

How many point counts are really needed depends on $g$ in an unexpected way. A
single point count never suffices, and for $g=2$ and $g=4$ the number $g$ is
optimal: the first $g-1$ counts do not determine the zeta function for all $A$.
For $g=3$, however, $g$ is \emph{not} optimal, the two counts $\#A(\F_q)$ and
$\#A(\F_{q^2})$ already determine the zeta function for every prime power
$q \geq 16$, so $g$ point counts are not always best possible
(\cref{prop:N3}). For $g \geq 5$ the exact value of $N(g)$ is open; computational evidence for small $q$ and the non-existence of distinct $\Z[q]$-families with matching point counts lead one to guess that $N(5) = N(6) = 4$ (\cref{ques:N56,ques:Ng}).

One may also ask about the asymptotics of $N(g)$: whether the limit
$\lim\limits_{g \to \infty} N(g)/g$ exists and, if not, what its lower and upper limits
are (\cref{ques:Ng}).

The constant
$Q(g)$ we obtain is super-exponential in $g$, and is far from optimal; we have not tried to optimize it. We discuss this and the question
of whether a constant polynomial in $g$ is possible in
\cref{sec:remarks}.

\subsection*{Outline of the argument}
Throughout we work with the power sums $s_n = \sum_{i=1}^{2g} \alpha_i^{-n}$ of the
inverse Frobenius eigenvalues; knowing $s_1, \dots, s_g$ is equivalent to knowing
$f_A$. The two key elementary facts are that $q^n s_n$ is an \emph{integer}
(\cref{lem:integrality}), so that $s_n$ can be recovered by rounding once it is
known to additive error $< \tfrac12 q^{-n}$, and that the point count
$\#A(\F_{q^n})$ determines the slightly perturbed quantity
$\sum_{j \geq 1} s_{jn}/j$ exactly (\cref{lem:counts-to-sums}). The Weil bounds
make the perturbation small but, for $n$ close to $g$, not quite small enough to
round directly. The argument splits into two ranges. For $1 \leq i \leq g/2$ the
point counts over $\F_{q^i}$ and $\F_{q^{2i}}$ are both available, and a
truncated M\"obius inversion recovers $s_i$ exactly (\cref{lem:small-i}). For
$g/2 < i \leq g$ only $\#A(\F_{q^i})$ is available; here we run an induction
(\cref{lem:induction-step}) that, using Newton's identities and the functional
equation, estimates $s_{2i}$ accurately enough from the previously determined
power sums to sharpen the estimate of $s_i$ to error $o(q^{-i})$ and recover it
by rounding.

\subsection*{Acknowledgments}
We thank the organizers of the IAS/PCMI Research Program 2022 ``Number Theory
informed by Computation'' and Kiran Kedlaya for helpful advice. The large language
model Claude Opus~4.8 was used as an assistant in checking the arguments of this
article, in writing the accompanying \textsc{Sage} script, and in carrying out the
Lean formalization discussed at the end of \cref{sec:remarks}; all mathematical
content remains the responsibility of the authors.

\section{Preliminaries on the Frobenius eigenvalues}\label{sec:prelim}

Throughout the article, we fix an embedding $\ol\Q \inj \C$. Let $A$ be an
abelian variety of dimension $g$ over a finite field $\F_q$. Let
$\alpha_i \in \ol\Q$, $1 \leq i \leq 2g$, be the roots of the characteristic
polynomial
\[
    f_A(T) = (1 - \alpha_1T)(1 - \alpha_2T)\cdot \ldots \cdot(1 - \alpha_{2g}T)
           = a_0 + a_1T + \ldots + a_{2g}T^{2g}
\]
of the Frobenius acting on the $\ell$-adic Tate module for some $\ell \nmid q$,
independent of $\ell$ and ordered in such a way that $\alpha_i\alpha_{i+g} = q$
for all $1 \leq i \leq g$. This is possible by Poincar\'e duality. Note that
$a_0 = 1$ and $a_{2g} = \prod_{i = 1}^{2g}\alpha_i = q^g$. By the Riemann
hypothesis for $A$, one has $|\alpha_i| = q^{1/2}$ for all
complex embeddings.

\begin{notation}\label[notation]{not:e-s}
Let
\[
    e_n \defeq \sum_{\substack{I \subseteq \{1,\ldots,2g\}\\\#I = n}}\prod_{i \in I}\alpha_i^{-1}
    \qquad (0 \leq n \leq 2g)
\]
be the elementary symmetric polynomial of degree $n$ in the $\alpha_i^{-1}$, and
\[
    s_n \defeq \sum_{i=1}^{2g}\alpha_i^{-n} \qquad (n \geq 1)
\]
their $n$-th power sum. We write $\beta_i \defeq \alpha_i^{-1}$, so
$|\beta_i| = q^{-1/2}$.
\end{notation}

Knowing $f_A$ is equivalent to knowing $e_1, \dots, e_{2g}$ (the $e_n$ are, up to
sign and the normalisation by $a_{2g} = q^g$, the coefficients $a_n$), and by
Newton's identities (\cref{lem:sym-identities}) equivalent to knowing
$s_1, \dots, s_{2g}$. The functional
equation will reduce this to $s_1, \dots, s_g$.

\begin{lemma}[Functional equation]\label[lemma]{lem:functional-eq}
For $1 \leq n \leq g$ one has
\begin{equation} \label{eq:complement of e}
    e_{g + n} = q^{-n}e_{g - n}.
\end{equation}
\end{lemma}

\begin{proof}
Let $\sigma$ be the involution $i \leftrightarrow i+g$, so that
$q/\alpha_i = \alpha_{\sigma(i)}$. For $I \subseteq \{1,\dots,2g\}$ with
$\#I = g+n$, using $\prod_{i=1}^{2g}\alpha_i^{-1} = q^{-g}$,
\[
    \prod_{i \in I}\alpha_i^{-1}
    = q^{-g}\prod_{i \notin I}\alpha_i
    = q^{-g}\prod_{i \notin I}\frac{q}{\alpha_{\sigma(i)}}
    = q^{-g}\,q^{\,g-n}\prod_{i \notin I}\alpha_{\sigma(i)}^{-1}
    = q^{-n}\prod_{i \notin I}\alpha_{\sigma(i)}^{-1}.
\]
As $I$ ranges over the subsets of size $g+n$, the set $\sigma(\{1,\dots,2g\}\setminus I)$
ranges over all subsets of size $g-n$. Summing therefore gives
$e_{g+n} = q^{-n}e_{g-n}$.
\end{proof}

\begin{lemma}[Weil bounds]\label[lemma]{lem:weil-bounds}
For all $n \geq 1$ and $0 \leq m \leq 2g$,
\begin{equation} \label{eq:bound s}
    |s_n| \leq 2g\,q^{-n/2}, \qquad |e_m| \leq \binom{2g}{m} q^{-m/2}.
\end{equation}
\end{lemma}

\begin{proof}
Immediate from $|\beta_i| = q^{-1/2}$: $s_n$ is a sum of $2g$ terms of modulus
$q^{-n/2}$, and $e_m$ is a sum of $\binom{2g}{m}$ products of $m$ such terms.
\end{proof}

\begin{lemma}[Integrality and rounding]\label[lemma]{lem:integrality}
For every $n \geq 1$,
\[
    q^n s_n = \sum_{i=1}^{2g}\alpha_i^{n} = \Tr\bigl(\Frob_q^n \mid \Het^1(A_{\ol\F_q}, \Q_\ell)\bigr) \in \Z.
\]
Consequently, if $\wt s$ is a known approximation of $s_n$ with
$|s_n - \wt s| < \tfrac12 q^{-n}$, then $s_n$ is determined: it equals $m q^{-n}$,
where $m$ is the unique integer with $|mq^{-n} - \wt s| < \tfrac12 q^{-n}$, namely the
nearest integer to $q^n \wt s$.
\end{lemma}

\begin{proof}
The functional equation $\alpha_i\alpha_{i+g} = q$ says that
$\{q\beta_i\}_{i} = \{q/\alpha_i\}_i = \{\alpha_i\}_i$ as multisets, whence
$q^n s_n = \sum_{i=1}^{2g} (q\beta_i)^n = \sum_{i=1}^{2g} \alpha_i^n$. The latter
is the $n$-th power sum of the roots of the monic integral polynomial
$T^{2g}f_A(1/T) = \prod_{i=1}^{2g}(T - \alpha_i)$, hence an integer (it equals the trace
of $\Frob_q^n$ on $\Het^1$). The rounding statement is clear: the values $mq^{-n}$ for distinct integers
$m$ are at distance $\geq q^{-n}$, so at most one lies
within $\tfrac12 q^{-n}$ of $s_n$, and it is found by rounding $q^n \wt s$.
\end{proof}

\begin{lemma}[Symmetric function identities]\label[lemma]{lem:sym-identities}
The power sums and elementary symmetric functions of the $\beta_i$ are related by
Newton's identities \cite[I.2, (2.11$'$)]{Macdonald1995}
\begin{equation} \label{eq:Newton-Girard}
    n\,e_n = \sum_{i = 1}^n (-1)^{i - 1}e_{n - i}\,s_i \qquad (1 \leq n \leq 2g),
\end{equation}
and, equivalently, by Macdonald's closed form \cite[I.2, (2.14$'$)]{Macdonald1995},
obtained by expanding the generating-function identity
\[
    \sum_{n\geq0} e_n t^n = \exp\Bigl(\sum_{k\geq1}\tfrac{(-1)^{k-1}}{k}s_k t^k\Bigr)
\]
(itself $\log\prod_i(1+\beta_i t) = \sum_{k\geq1}\tfrac{(-1)^{k-1}}{k}s_k t^k$ exponentiated, where $\prod_i(1+\beta_i t) = \sum_n e_n t^n$):
\begin{equation} \label{eq:en by si}
    e_n = (-1)^n\sum_{\substack{m_1 + 2m_2 + \ldots + nm_n = n\\m_k \geq 0}}\;\prod_{k = 1}^n\frac{(-s_k)^{m_k}}{m_k!\,k^{m_k}}.
\end{equation}
In particular, knowing $s_1, \dots, s_m$ exactly determines $e_1, \dots, e_m$
exactly, and conversely.
\end{lemma}

\begin{proof}
These are standard \cite[I.2, (2.11$'$) and (2.14$'$)]{Macdonald1995}. Equating
coefficients of $t^n$ in the logarithmic derivative of the generating function
$\prod_{i=1}^{2g}(1+\beta_i t) = \sum_{n=0}^{2g} e_n t^n$ gives \eqref{eq:Newton-Girard};
expanding the exponential generating function (2.14$'$), equivalently Macdonald's
closed form $e_n = \sum_{\lambda \vdash n}\eps_\lambda\,z_\lambda^{-1}\prod_k s_k^{m_k}$
(sum over partitions $\lambda = (1^{m_1}2^{m_2}\cdots)$ of $n$) with
$z_\lambda = \prod_k k^{m_k}m_k!$ and $\eps_\lambda = (-1)^{n - \ell(\lambda)}$,
gives \eqref{eq:en by si}.
\end{proof}

We record the elementary error-propagation rule used repeatedly below: if
$x = \wt x + \Delta x$ and $y = \wt y + \Delta y$ with $|x - \wt x| \leq \Delta x$
and $|y - \wt y| \leq \Delta y$ (we use $\Delta(\,\cdot\,)$ both for the error and
for a bound on it), then
\[
    |xy - \wt x\,\wt y| \leq |x|\,\Delta y + |y|\,\Delta x + \Delta x\,\Delta y,
    \qquad |x + y - \wt x - \wt y| \leq \Delta x + \Delta y.
\]
If $x$ is known exactly we take $\Delta x = 0$.

\section{From point counts to power sums}\label{sec:counts}

Let
\[
    c_n \defeq \#A(\F_{q^n}) = \prod_{i=1}^{2g}(1 - \alpha_i^n)
\]
be the $n$-th point count of $A$.

\begin{lemma}[Point counts determine perturbed power sums]\label[lemma]{lem:counts-to-sums}
For all $n \geq 1$,
\begin{equation}\label{eq:cn_to_sn}
    \log\frac{q^{gn}}{c_n} = \sum_{j \geq 1}\frac{s_{jn}}{j},
\end{equation}
and, by M\"obius inversion,
\begin{equation}\label{eq:mobius}
    s_n = \sum_{j \geq 1}\frac{\mu(j)}{j}\,\log\frac{q^{gjn}}{c_{jn}}.
\end{equation}
\end{lemma}

\begin{proof}
Factoring out the dominant term and using $\prod_{i=1}^{2g} \alpha_i^n = q^{gn}$,
\[
    c_n = \prod_{i=1}^{2g}(1-\alpha_i^n) = \Bigl(\prod_{i=1}^{2g}\alpha_i^n\Bigr)\prod_{i=1}^{2g}(1-\beta_i^n)
        = q^{gn}\prod_{i=1}^{2g}(1-\beta_i^n),
\]
so that, since $|\beta_i^n| = q^{-n/2} < 1$,
\[
    \log\frac{q^{gn}}{c_n} = -\sum_{i=1}^{2g} \log(1-\beta_i^n)
        = \sum_{i=1}^{2g}\sum_{j\geq1}\frac{\beta_i^{jn}}{j} = \sum_{j\geq1}\frac{s_{jn}}{j}.
\]
This is \eqref{eq:cn_to_sn}, which is \cite[\S8]{Kedlaya2006} rewritten in our
normalisation (Kedlaya works with $q^{-gn}c_n$ and $\tfrac1n s_n$ in place of $c_n$ and
$s_n$).

For \eqref{eq:mobius}, write
$L_m \defeq \log(q^{gm}/c_m) = \sum_{k\geq1} s_{km}/k$ and compute
\[
    \sum_{j\geq1}\frac{\mu(j)}{j}L_{jn}
    = \sum_{j\geq1}\frac{\mu(j)}{j}\sum_{k\geq1}\frac{s_{jkn}}{k}
    = \sum_{m\geq1}\frac{s_{mn}}{m}\sum_{j\mid m}\mu(j)
    = s_n,
\]
the inner sum being $1$ if $m=1$ and $0$ otherwise. All series converge
absolutely by \eqref{eq:bound s}.
\end{proof}

We abbreviate $L_m \defeq \log(q^{gm}/c_m)$; thus $L_m$ is known exactly whenever
$m \leq g$, and $s_m = L_m - \sum_{j\geq2} s_{jm}/j$.

\begin{lemma}[Seed estimate]\label[lemma]{lem:seed}
For $q \geq 2$ and $i \geq 1$,
\begin{equation}\label{eq:seed}
    |L_i - s_i| = \Bigl|\sum_{j\geq2}\frac{s_{ji}}{j}\Bigr| \leq 4g\,q^{-i}.
\end{equation}
\end{lemma}

\begin{proof}
Using \eqref{eq:bound s} and extracting the geometric series,
\[
    \Bigl|\sum_{j\geq2}\frac{s_{ji}}{j}\Bigr|
    \leq \sum_{j\geq2}\frac{2g\,q^{-ji/2}}{j}
    \leq g\,q^{-i}\sum_{j\geq0}q^{-ji/2}
    = g\,q^{-i}\frac{1}{1-q^{-i/2}}
    \leq 4g\,q^{-i},
\]
where the last step uses $q \geq 2$, $i \geq 1$, so $1-q^{-i/2}\geq 1 - 2^{-1/2} \geq \tfrac14$.
\end{proof}

Thus $L_i$ already determines $s_i$ to within $4g\,q^{-i}$. To round by
\cref{lem:integrality} we would need the error below $\tfrac12 q^{-i}$, which
\eqref{eq:seed} does not give; sharpening it is the content of the next two
sections.

\section{The small range \texorpdfstring{$i \leq g/2$}{i <= g/2}}\label{sec:small}

\begin{lemma}\label[lemma]{lem:small-i}
Let $1 \leq i \leq g/2$ and suppose $q > 256\,g^2$. Then the truncated M\"obius sum
\[
    \wt s_i \defeq \sum_{\substack{j\geq1\\ ji \leq g}}\frac{\mu(j)}{j}\,L_{ji},
\]
which depends only on $c_1, \dots, c_g$, satisfies $|s_i - \wt s_i| < \tfrac12 q^{-i}$.
In particular $s_i$ is determined by $c_1, \dots, c_g$.
\end{lemma}

\begin{proof}
By \eqref{eq:mobius}, $s_i - \wt s_i = \sum_{j:\,ji>g}\frac{\mu(j)}{j}L_{ji}$. By
\eqref{eq:bound s}, for $m \geq 1$,
\[
    |L_m| \leq \sum_{k\geq1}\frac{2g\,q^{-km/2}}{k} \leq \frac{2g\,q^{-m/2}}{1-q^{-m/2}} \leq 4g\,q^{-m/2},
\]
where the last step uses the hypothesis $q > 256\,g^2$, so that $q^{-m/2} \leq q^{-1/2} < \tfrac12$ and hence $1 - q^{-m/2} > \tfrac12$.
Let $j_0$ be the smallest $j$ with $ji>g$; then
$j_0 i \geq g+1$, and
\[
    |s_i - \wt s_i| \leq \sum_{j\geq j_0}|L_{ji}|
    \leq 4g\sum_{j\geq j_0}q^{-ji/2}
    \leq 4g\,\frac{q^{-j_0 i/2}}{1-q^{-i/2}}
    \leq 8g\,q^{-(g+1)/2}.
\]
The last step uses the hypothesis $q > 256\,g^2$ (so that $q^{-i/2} \leq q^{-1/2} < \tfrac12$) to
bound the geometric factor by $\frac{1}{1-q^{-i/2}} \leq 2$.
We need $8g\,q^{-(g+1)/2} < \tfrac12 q^{-i}$, i.e. $q^{(g+1)/2 - i} > 16g$. Since
$i \leq g/2$ we have $(g+1)/2 - i \geq \tfrac12$, so $q^{1/2} > 16g$, i.e.
$q > 256\,g^2$, suffices. The conclusion follows from \cref{lem:integrality}.
\end{proof}

\section{The large range \texorpdfstring{$g/2 < i \leq g$}{g/2 < i <= g}}\label{sec:large}

Here $c_{2i}$ is no longer available ($2i > g$), and \eqref{eq:seed} is too weak
to recover $s_i$. We argue by induction on $i$, sharpening the seed estimate of
$s_{2i}$ first.

Fix the constant
\[
    D \defeq 16\,g^3 p(2g),
\]
where $p(\,\cdot\,)$ is the partition function, and set
\[
    B_j \defeq g\,q^{-j/2 - g/4}\,D^{\,j-g} \qquad (j \geq 1).
\]
We will propagate the bound 
\begin{align}
    \label{eq:propagatingerrorbound}
    \Delta s_j \leq B_j
\end{align}
from the range $j \in [1,g]$ to the
range $j \in [1,2g]$.
Note that, for a fixed $g$ and large enough $q$, $B_j$ is decreasing
in $j$ as $\sim (Dq^{-1/2})^j$, while the error $\Delta s_j$ of the seed estimate (\ref{eq:seed}) coming from the $j^{th}$ point count decreases as $\sim q^{-j}$. So for $1 \leq j \leq g$, where we do have point count information, the error bound (\ref{eq:propagatingerrorbound}) holds easily:
indeed, for $1 \leq j \leq g/2$ we have $\Delta s_j = 0$ by
\cref{lem:small-i}, so the bound is trivial; for $g/2 < j \leq g$, this is shown in \cref{lem:base} which forms the base case of
our inductive method. The induction step (\cref{lem:induction-step}) then
carries it from $[1,g]$ up to the desired range $[1,2g]$.

\subsection{Expressing the odd and even doubled power sums}\label{ss:formula}

\begin{lemma}\label[lemma]{lem:formula}
Let $g/2 < i \leq g$ and write $m \in \{2i-1, 2i\}$, assuming $m > g$, and set
$n \defeq m - g > 0$. Then, isolating the top term $s_m$ in \eqref{eq:en by si}
for $e_m$ and substituting $e_m = q^{-n}e_{g-n}$ from \eqref{eq:complement of e},
\begin{equation}\label{eq:sm-formula}
    \frac{s_m}{m} = -q^{-n}\!\!\sum_{\substack{\sum_k k m_k = 2g-m\\ m_k\geq0}}\prod_{k=1}^{2g-m}\frac{(-s_k)^{m_k}}{m_k!\,k^{m_k}}
    \;+\!\!\sum_{\substack{\sum_k k m_k = m\\ m_k\geq0,\ m_m = 0}}\prod_{k=1}^{m-1}\frac{(-s_k)^{m_k}}{m_k!\,k^{m_k}}.
\end{equation}
\end{lemma}

\begin{proof}
In \eqref{eq:en by si} for $e_m$, the partition with $m_m = 1$ (all other
$m_k=0$) contributes $(-1)^m\frac{-s_m}{m} = (-1)^{m-1}\frac{s_m}{m}$; separating
it off gives
\[
    e_m = (-1)^{m-1}\frac{s_m}{m} + (-1)^m\!\!\sum_{\substack{\sum k m_k = m\\ m_m=0}}\prod_{k<m}\frac{(-s_k)^{m_k}}{m_k!\,k^{m_k}}.
\]
Solving for $s_m/m$ and replacing $e_m$ by $q^{-n}e_{g-n}$, with $e_{g-n}$
expanded by \eqref{eq:en by si} (here $g-n = 2g-m$), yields
\eqref{eq:sm-formula}.
\end{proof}

\subsection{The inductive estimate}\label{ss:induction}

\begin{lemma}[Induction step]\label[lemma]{lem:induction-step}
Let $g/2 < i \leq g$. Suppose
\begin{enumerate}
    \item $s_j$ is known exactly for all $j \leq i-1$, and
    \item $\Delta s_j \leq B_j$ for all $i \leq j \leq 2i-2$.
\end{enumerate}
Then $\Delta s_{2i} \leq B_{2i}$, and also $\Delta s_{2i-1} \leq B_{2i-1}$ provided
$2i-1 > g$.
\end{lemma}

\begin{proof}
We treat $m = 2i-1$ in full and then indicate the change for $m = 2i$. In both cases
$m > g$ (for $m = 2i-1$ this is the standing assumption of the lemma, and for $m = 2i$ it
holds since $i > g/2$), so we may set $n \defeq m - g > 0$ and apply \cref{lem:formula}.
In \eqref{eq:sm-formula} the factors $s_k$ with $k \leq i-1$ are
exact, so only partitions containing some index $k_0 \geq i$ contribute to the
error.

\emph{At most one large index.} If two factors had indices $\geq i$, their
contribution to $\sum_k k m_k$ would be $\geq 2i > 2i-1 = m$ for the second
(main) sum, and $\geq 2i > g \geq 2g-m$ for the first sum; both are impossible
for $m = 2i-1$. Hence at most one index $k_0 \geq i$ occurs, necessarily with
$m_{k_0}=1$. Applying the error-propagation rule from the end of \cref{sec:prelim} to the
products and bounding the remaining (exact)
factors by \eqref{eq:bound s},
\begin{align*}
    \Delta s_{2i-1}
    \leq (2i-1)\Biggl[&q^{-n}\!\!\sum_{k_0=i}^{2g-2i+1}\frac{\Delta s_{k_0}}{k_0}\!\!\sum_{\substack{\sum k m_k = 2g-2i+1-k_0}}\prod_{k}\frac{|s_k|^{m_k}}{m_k!\,k^{m_k}}\\
    &+\sum_{k_0=i}^{2i-2}\frac{\Delta s_{k_0}}{k_0}\sum_{\substack{\sum k m_k = 2i-1-k_0}}\prod_{k}\frac{|s_k|^{m_k}}{m_k!\,k^{m_k}}\Biggr].
\end{align*}
Ignoring the denominators $k_0$ and $m_k!\,k^{m_k}$ (all $\geq1$), using
$|s_k|\leq 2g\,q^{-k/2}$ and $\sum_k m_k \leq \sum_k k m_k$, and bounding the
number of partitions of $M$ by $p(M) \leq p(2g)$, each inner sum over partitions
of $M$ is at most $p(2g)\,(2g\,q^{-1/2})^{M}$. Therefore
\begin{equation}\label{eq:ind-bracket}
\begin{aligned}
    \Delta s_{2i-1}
    \leq (2i-1)p(2g)\Biggl[&q^{-n}\sum_{k_0=i}^{2g-2i+1}\Delta s_{k_0}\,(2g\,q^{-1/2})^{2g-2i+1-k_0}\\
    &+\sum_{k_0=i}^{2i-2}\Delta s_{k_0}\,(2g\,q^{-1/2})^{2i-1-k_0}\Biggr].
\end{aligned}
\end{equation}
Since $2g-2i+1 + 2n = 2i-1$, one checks
\[
    q^{-n}(2g\,q^{-1/2})^{2g-2i+1-k_0} = (2g)^{2g-2i+1-k_0}q^{-(2i-1-k_0)/2},
\]
which
is term-by-term $\leq (2g\,q^{-1/2})^{2i-1-k_0}$ because $2g-2i+1 = 2g-(2i-1) < g < 2i-1$. Moreover the index range $i \leq k_0 \leq 2g-2i+1$
of the first sum is contained in that of the second. Hence the bracket is at most twice
the second sum, and using $2(2i-1) \leq 4g$,
\[
    \Delta s_{2i-1} \leq 4g\,p(2g)\sum_{k_0=i}^{2i-2}\Delta s_{k_0}\,(2g\,q^{-1/2})^{2i-1-k_0}.
\]
Now insert the induction hypothesis $\Delta s_{k_0} \leq B_{k_0} = g\,q^{-k_0/2-g/4}D^{k_0-g}$.
The powers of $q$ combine to $q^{-k_0/2-g/4}\cdot q^{-(2i-1-k_0)/2} = q^{-(2i-1)/2-g/4}$,
independently of $k_0$, so
\begin{equation}\label{eq:ind-geom}
    \Delta s_{2i-1}
    \leq 4g^2 p(2g)\,q^{-(2i-1)/2-g/4}\sum_{k_0=i}^{2i-2}D^{\,k_0-g}(2g)^{2i-1-k_0}.
\end{equation}
Writing the summand as $D^{-g}(2g)^{2i-1}(D/2g)^{k_0}$ and using $D \geq 2g$, the
sum is increasing in $k_0$, so it is at most
$(i-1)\,D^{2i-2-g}(2g)^{2i-1-(2i-2)} = (i-1)(2g)\,D^{2i-2-g} \leq 2g^2 D^{2i-2-g}$.
Hence
\[
    \Delta s_{2i-1} \leq 8g^4 p(2g)\,q^{-(2i-1)/2-g/4}D^{2i-2-g}
    = \frac{8g^3 p(2g)}{D}\,B_{2i-1} \leq \tfrac12\,B_{2i-1},
\]
since $D = 16g^3 p(2g)$. This proves the bound for $m = 2i-1$ with a factor of
$2$ to spare.

\emph{The case $m = 2i$.} The same computation applies, with $2g-2i$ in place of
$2g-2i+1$ and $2i$ in place of $2i-1$, and yields a contribution
$\leq \tfrac12 B_{2i}$ from the partitions with a single large index, by the
identical chain of inequalities (the $q$-exponent is now $q^{-i-g/4}$). The one
new feature is that $\sum_k k m_k = 2i$ admits a partition with \emph{two} large
indices, namely $m_i = 2$. Its contribution to $e_{2i}$ is
$\frac{(-s_i)^2}{2!\,i^2}$, with error at most
$\frac{1}{i^2}\bigl(|s_i|\,\Delta s_i + \tfrac12(\Delta s_i)^2\bigr)$; multiplied
by $2i$ and bounded via $|s_i|\leq 2g\,q^{-i/2}$ and $\Delta s_i \leq B_i$ this is
at most
\[
    \frac{2}{i}\Bigl(2g\,q^{-i/2}B_i + \tfrac12 B_i^2\Bigr)
    \leq 4g\,q^{-i/2}B_i + B_i^2 \leq 6g\,q^{-i/2}B_i,
\]
where the final step uses $B_i \leq 2g\,q^{-i/2}$. This holds directly from the
definition of $B_i$: since $i \leq g$ we have $q^{-g/4} \leq 1$ and $D^{i-g} \leq 1$, so
$B_i = g\,q^{-i/2-g/4}D^{i-g} \leq g\,q^{-i/2} \leq 2g\,q^{-i/2}$. Now
$q^{-i/2}B_i = g\,q^{-i-g/4}D^{i-g}$, so this contribution is
$\leq 6g^2 q^{-i-g/4}D^{i-g} = \frac{6g}{D^{i}}B_{2i} \leq \tfrac12 B_{2i}$
because $D \geq 16g$. Adding the two contributions gives $\Delta s_{2i}\leq B_{2i}$.
\end{proof}

\begin{lemma}[Base case]
\label[lemma]{lem:base}
Suppose $q \geq D^{2g}$. Then for every integer $j$ with $g/2 < j \leq g$ the seed
estimate satisfies $4g\,q^{-j} \leq B_j$.
\end{lemma}

\begin{proof}
Note $D = 16g^3p(2g) \geq 16$, so $\log_D 4 \leq \tfrac12$. The asserted
inequality $4g\,q^{-j} \leq g\,q^{-j/2-g/4}D^{\,j-g}$ is equivalent to
$4\,D^{\,g-j} \leq q^{(2j-g)/4}$. Put $Q \defeq \log_D q \geq 2g$ and
$t \defeq 2j - g$; since $j$ is an integer with $g/2 < j \leq g$ we have
$1 \leq t \leq g$ and $g - j = (g-t)/2$. Taking $\log_D$, the inequality reads
\[
    \tfrac{t}{4}\,Q \;\geq\; \log_D 4 + \tfrac{g-t}{2}.
\]
Using $Q \geq 2g$, the left side is $\geq \tfrac{t}{4}\cdot 2g = \tfrac{gt}{2}$, so
it suffices that $\tfrac{gt}{2} \geq \tfrac12 + \tfrac{g-t}{2}$, i.e.
$g(t-1) + t \geq 1$, which holds for all $t \geq 1$. Hence
$4\,D^{g-j} \leq q^{(2j-g)/4}$, as required.
\end{proof}

\section{Proof of the main theorem}\label{sec:proof}

\begin{proof}[Proof of \cref{thm:main}]
As noted after the statement we may assume $g \geq 2$, the case $g = 1$ being
classical. Take $Q(g) = D^{2g+2} = \bigl(16\,g^3 p(2g)\bigr)^{2g+2}$ and assume
$q > Q(g)$. In
particular $q > 256\,g^2$, so \cref{lem:small-i} applies, and $q > D^{2g}$, so
\cref{lem:base} applies. We show by induction on
$i$ that $s_i$ is determined by $c_1,\dots,c_g$ for all $1 \leq i \leq g$; this
gives $s_1,\dots,s_g$, hence by \eqref{eq:complement of e} and
\cref{lem:sym-identities} all of $e_1,\dots,e_{2g}$, i.e. $f_A$.

\emph{Range $i \leq g/2$.} \cref{lem:small-i} determines each such $s_i$ exactly.

\emph{Range $g/2 < i \leq g$.} We process these $i$ in increasing order. Fix such
an $i$ and assume inductively that $s_j$ is known exactly for all $j \leq i-1$
(true: for $j \leq g/2$ by \cref{lem:small-i}, and for $g/2 < j < i$ by previous
steps of this induction). We also have $\Delta s_j \leq B_j$ for $i \leq j \leq 2i-2$:
for $g/2 < j \leq g$ this is the seed estimate \eqref{eq:seed} together with
\cref{lem:base}, and for $g < j \leq 2i-2$ it was produced by the step with index
$\lceil j/2\rceil < i$ via \cref{lem:induction-step}. Thus the hypotheses of
\cref{lem:induction-step} hold, and we obtain $\Delta s_{2i} \leq B_{2i}$.

Now recover $s_i$. By \cref{lem:counts-to-sums},
$s_i = L_i - \tfrac12 s_{2i} - \sum_{j\geq3} s_{ji}/j$, where $L_i$ is known
exactly (as $i \leq g$) and $s_{2i}$ is known with error $\leq B_{2i}$.
Estimating the tail by \eqref{eq:bound s},
$\bigl|\sum_{j\geq3}s_{ji}/j\bigr| \leq \tfrac13\sum_{j\geq3}2g\,q^{-ji/2}
\leq 2g\,q^{-3i/2}/(1-q^{-i/2}) \leq 4g\,q^{-3i/2}$,
the total error in the resulting approximation $\wt s_i \defeq L_i - \tfrac12 s_{2i}$
(with $s_{2i}$ the known approximation) is
\[
    \Delta s_i \leq \tfrac12 B_{2i} + 4g\,q^{-3i/2}.
\]
For the first term, $\tfrac12 B_{2i} = \tfrac12 g\,q^{-i-g/4}D^{2i-g} < \tfrac14 q^{-i}$
is equivalent to $2g\,D^{2i-g} < q^{g/4}$; since $2i-g \leq g$ this follows from
$2g\,D^{g} < q^{g/4}$, i.e. $q > (2g)^{4/g}D^4$. As $D \geq 2g$ we have
$(2g)^{4/g}D^4 \leq D^{4/g}D^4 \leq D^6 \leq D^{2g+2}$ for $g \geq 2$, so this
holds. For the second term, $4g\,q^{-3i/2} < \tfrac14 q^{-i}$ is $q^{i/2} > 16g$,
which holds since $i \geq 1$ and $q > 256\,g^2$. Hence $\Delta s_i < \tfrac12 q^{-i}$,
and \cref{lem:integrality} recovers $s_i$ exactly.

This completes the induction, and with it the proof.
\end{proof}

The proof is constructive: it yields an algorithm that, on input $q$, $g$, and
the point counts $c_1, \dots, c_g$, returns the $L$\nobreakdash-polynomial $f_A$
(equivalently $s_1, \dots, s_g$) by the small- and large-range recoveries above,
the only non-algebraic operations being the evaluation of the logarithms
$L_m = \log(q^{gm}/c_m)$ to sufficient precision and the final roundings. A
\textsc{Sage} implementation is available in the accompanying script
\texttt{recover\_zeta.sage}. It performs $O(g^3)$ arithmetic operations on
numbers of $\Theta(g\log q)$ bits, for a bit complexity of
$\widetilde{O}(g^4\log q)$ with FFT-based multiplication; in particular, for
fixed $g$ it is quasi-linear in $\log q$, hence polynomial in the input size.

\section{Optimality and remarks}\label{sec:remarks}

How many of the point counts are really needed? Write $N(g)$ for the least $N$
such that, for all sufficiently large $q$, the point counts
\[
    \#A(\F_{q^i}), \qquad 1 \leq i \leq N,
\]
determine the isogeny class of \emph{every} abelian variety $A$ of dimension $g$
over $\F_q$. \cref{thm:main} says $N(g) \leq g$. One might expect equality, the
``$g$ unknowns, so $g$ measurements'' heuristic. It is, however, \emph{false}: $g$ is optimal for
$g=2$ and $g=4$, but for $g=3$ already two point counts suffice.

We first record some basic facts. The first is
\emph{monotonicity}: $N(g) \leq N(g+1)$, so a lower bound in one dimension
propagates to all higher ones. Indeed, if two non-isogenous $g$-dimensional abelian varieties $A$ and $B$ over
$\F_q$ share their first $N$ point counts, then for any elliptic curve $E/\F_q$ the
products $A \times E$ and $B \times E$
still share their first $N$ counts (point counts are multiplicative on products) and
they remain non-isogenous.

The mechanism underlying the discussion in this section is that a point count is a product of values of
$f_A$ at certain roots of unity:
\begin{equation}\label{eq:cn-as-eval}
    c_n = \#A(\F_{q^n}) = \prod_{i=1}^{2g}(1-\alpha_i^n) = \prod_{\zeta^n = 1} f_A(\zeta).
\end{equation}
For $n \leq 2$ these roots are $\pm 1$ and the resulting
conditions are \emph{linear} in the coefficients, which is what makes the fibres of the map $f_A \mapsto (c_1, c_2, \ldots, c_N)$ amenable to a direct lattice analysis. For $n > 2$, the resulting equations have degree $\phi(n) > 1$ making them hard to handle.

Throughout we use that the free coefficients are $(a_1,\dots,a_g)$, the remaining
$a_k$ being fixed by $a_0=1$ and the functional equation $a_{2g-k}=q^{g-k}a_k$,
and that the $q$-Weil polynomials form a subset of the integer lattice points inside the Weil region in $\Z^g$,
the body cut out by the Riemann hypothesis, whose $a_k$-extent is
$|a_k| \leq \binom{2g}{k}q^{k/2}$.

We also use the \emph{trace polynomial}. By the functional equation the reciprocal roots
of $f_A$ pair as $\alpha,\,q/\alpha$, so the characteristic polynomial
$P(T) = T^{2g}f_A(1/T) = \prod_{i=1}^{2g}(T-\alpha_i)$ factors over $\Qbar$ as
$P(T) = \prod_{i=1}^{g}(T^2 - \beta_i T + q)$ with $\beta_i = \alpha_i + q/\alpha_i$, and
$h(T) \defeq \prod_{i=1}^{g}(T-\beta_i)$ has integer coefficients. Since $T^2-\beta T+q$
has both roots of modulus $\sqrt q$ exactly when $\beta \in [-2\sqrt q, 2\sqrt q]$, the
polynomial $P$ is $q$-Weil if and only if all $\beta_i$ lie in $[-2\sqrt q, 2\sqrt q]$.
We also get the relation $P(T) = T^g h(T+\frac{q}{T})$.
Finally $\#A(\F_q) = P(1) = \prod_{i=1}^g\bigl((1+q)-\beta_i\bigr) = h(1+q)$, so a
coincidence $\#A(\F_q) = \#B(\F_q)$ of first point counts is the single equation
$h_A(1+q) = h_B(1+q)$.

\begin{lemma}\label[lemma]{lem:one-count}
For every $g \geq 2$ a single point count never suffices: $N(g) > 1$.
In particular, $N(2) = 2$.
\end{lemma}
\begin{proof}[Proof of \cref{lem:one-count}]
Due to monotonicity of $N(g)$, it is enough to show $N(2) > 1$.

We argue by an explicit family, valid
for every odd prime power $q$. In the trace-polynomial coordinates above take
\[
    h_A(T) = T^2 - (q+2), \qquad h_B(T) = T^2 - T - 1,
\]
whose roots, $\pm\sqrt{q+2}$ (with $q+2 \leq 4q$) and $\tfrac{1\pm\sqrt5}{2}$ (of modulus
$< 2 \leq 2\sqrt q$), all lie in $[-2\sqrt q, 2\sqrt q]$; the corresponding
\[
    P_A(T) = T^4 + (q-2)T^2 + q^2, \qquad P_B(T) = T^4 - T^3 + (2q-1)T^2 - qT + q^2
\]
are therefore $q$-Weil, and ordinary (their middle coefficients $q-2$ and $2q-1$ are prime
to $p$ for odd $q$). They are distinct, hence non-isogenous, yet
$\#A(\F_q) = h_A(1+q) = h_B(1+q) = q^2+q-1$. Thus $N(2) > 1$. Since $N(2) \leq 2$ by \cref{thm:main}, we get that $N(2) = 2$.
And indeed, the second count separates the two families above: $c_2(A)-c_2(B) = -2(q+1)(q^2+q-1) \neq 0$. These claims are checked symbolically in \texttt{verify\_g2\_family.sage}, in \cite{Repo}.
\end{proof}

\begin{remark}
    For $g \geq 3$, one can also give a direct proof of \cref{lem:one-count} based on asymptotic counts. A single point count takes few values: by the Weil bounds
    $\#A(\F_q) = \prod_i(1-\alpha_i)$ lies in the interval
    $[(\sqrt q-1)^{2g},(\sqrt q+1)^{2g}]$, which contains only
    \[
        (\sqrt q+1)^{2g}-(\sqrt q-1)^{2g} = 4g\,q^{(2g-1)/2}\bigl(1+O(q^{-1})\bigr)
    \]
    integers. The isogeny classes it must separate are far more numerous: by
    \cite[Theorem~1.1]{DiPippoHowe1998} the number of isogeny classes of $g$-dimensional
    abelian varieties over $\F_q$ is
    \[
        \#\mathcal I(q,g) \sim v_g\,\frac{\varphi(q)}{q}\,q^{g(g+1)/4}, \qquad
        v_g = \frac{2^g}{g!}\prod_{j=1}^{g}\Bigl(\frac{2j}{2j-1}\Bigr)^{g+1-j},
    \]
    as $q\to\infty$ over prime powers. For $g \geq 3$ the exponent $g(g+1)/4$ exceeds
    $(2g-1)/2$ by $(g-1)(g-2)/4 > 0$, and since $\varphi(q)/q > \tfrac{1}{2}$, the ratio $\#\mathcal I(q,g)\big/\bigl(4g\,q^{(2g-1)/2}\bigr)$ tends to
    infinity as $q \rightarrow \infty$. Hence for $q$ large there are more isogeny classes than available values of
    $\#A(\F_q)$, and so there must exist two distinct isogeny classes which share their first point count; so $N(g) > 1$.

    For $g=2$ the two exponents coincide ($g(g+1)/4 = (2g-1)/2 = \tfrac32$) and the factor
    $\varphi(q)/q$ makes this comparison inconclusive (for $q=2^k$ it gives
    $v_2\,\varphi(q)/q = \tfrac{16}{3} < 8$).
\end{remark}

\begin{example}
    For $g=3$, the lower bound $N(g) > 1$ of \cref{lem:one-count} too is witnessed by an
    explicit family, valid for every odd prime power $q$ at once. In trace-polynomial
    coordinates take
    \[
        h_A(T) = T^3 - 3qT - 1, \qquad h_B(T) = T^3 - (3q-1)T - (q+2);
    \]
    for $q \geq 2$ the values of each at $-2\sqrt q,\,-\sqrt q,\,\sqrt q,\,2\sqrt q$ alternate
    in sign as $-,+,-,+$, placing one root in each subinterval of $[-2\sqrt q, 2\sqrt q]$, so
    the corresponding
    \[
        P_A(T) = T^6 - T^3 + q^3, \qquad P_B(T) = T^6 + T^4 - (q+2)T^3 + qT^2 + q^3
    \]
    are $q$-Weil, and ordinary for odd $q$ (their middle coefficients are $-1$ and $-(q+2)$).
    They are distinct, hence non-isogenous, yet both have $\#A(\F_q) = h(1+q) = q^3$; the
    second count separates them, $\#A(\F_{q^2}) = q^6+2q^3$ for $A$ against $q^6+2q^4+4q^3$ for
    $B$, in accordance with \cref{prop:N3}. For instance $q=5$ gives the abelian threefold isogeny classes with LMFDB labels \href{https://beta.lmfdb.org/Variety/Abelian/Fq/3/5/a_a_ab}{3.5.a\_a\_ab} and \href{https://beta.lmfdb.org/Variety/Abelian/Fq/3/5/a_b_ah}{3.5.a\_b\_ah}, and corresponding Weil polynomials $T^6 - T^3 + 125$ and $T^6 + T^4 - 7T^3 + 5T^2 + 125$, both with $\#A(\F_5) = 125$.
    These claims are checked symbolically in \texttt{verify\_g3\_family.sage} in \cite{Repo}.
\end{example}

As anticipated in the introduction (see also \cref{rem:bitcount}), the
``$g$ unknowns, so $g$ measurements'' heuristic already breaks down at $g = 3$.

\begin{proposition}[$g = 3$]\label[proposition]{prop:N3}
$N(3) = 2$: for abelian threefolds the two point counts $\#A(\F_q)$ and
$\#A(\F_{q^2})$ already determine the isogeny class for every prime power
$q \geq 16$. In particular $g$ point counts are \emph{not} optimal in general.
\end{proposition}

\begin{proof}[Proof of \cref{prop:N3}]
By \eqref{eq:cn-as-eval}, $c_2 = f_A(1)\,f_A(-1)$, so for $g=3$ knowing $(c_1,c_2)$
is the same as knowing the two integer linear functionals $f_A(1)=c_1$ and
$f_A(-1)=c_2/c_1$. Their common fibre is the rational line
\[
    \bigl\{\, (a_1,a_2,a_3) : f_A(1)=c_1,\ f_A(-1)=c_2/c_1 \,\bigr\},
\]
with primitive direction $v = (1,\,0,\,-(1+q^2))$ in the coordinates
$(a_1,a_2,a_3)$ (the common kernel of $\nabla f_A(\pm 1)$). We pass to the
trace-polynomial coordinates $(b_1,b_2,b_3)$, the coefficients of
$h(T)=T^3-b_1T^2+b_2T-b_3$.
Using the relation $P(T) = T^g h(T+\frac{q}{T})$ to compare coefficients with $P(T)=\sum_k a_kT^{6-k}$ gives the unimodular integral substitution
\[
    a_1=-b_1,\qquad a_2=b_2+3q,\qquad a_3=-b_3-2q\,b_1,
\]
so that $(b_1,b_2,b_3)$ ranges over the lattice $\Z^3$ as well, and $v$ becomes
$(-1,\,0,\,(1+q)^2)$. The integer points of the fibre form the arithmetic
progression obtained by adding multiples of $v$; consecutive ones differ by
$(1+q)^2$ in the coordinate $b_3$. Both must lie in the Weil region, and there the
\emph{product} bound $|b_3|=\bigl|\textstyle\prod_i\beta_i\bigr|\leq(2\sqrt q)^3=8\,q^{3/2}$
is much sharper than the symmetric-function bound
$|a_3|\leq\binom{6}{3}q^{3/2}=20\,q^{3/2}$, so two such points require only
$(1+q)^2\leq 16\,q^{3/2}$ instead of $1+q^2\leq 40\,q^{3/2}$. The former fails for
$q\geq 252$ (against $q\geq 1600$ for the latter), so for every prime power
$q\geq 252$ each fibre contains at most one Weil polynomial and $(c_1,c_2)$ is
injective on isogeny classes. The remaining prime powers $16\leq q\leq 251$ are
settled by the exhaustive enumeration of \cref{rem:ExhaustiveEnumeration};
Therefore the point counts $\#A(\F_q)$ and $\#A(\F_{q^2})$ determine the isogeny class of an abelian threefold for every prime power $q\geq 16$. Together with \cref{lem:one-count} this gives $N(3)=2$.
\end{proof}

The exhaustive enumeration of $q$-Weil polynomials that settles the range
$16 \leq q \leq 251$ also pins down the exact crossover: $(c_1,c_2)$ is non-injective
for every prime power $q \leq 13$ and injective for every prime power
$16 \leq q \leq 251$, so the analytic threshold $q \geq 252$ above and the enumerated
range meet with no gap.
For $16 \leq q \leq 25$, our enumeration also
agrees with the LMFDB \cite{lmfdb} tables, which are complete in that range for $g=3$. See
\cref{rem:ExhaustiveEnumeration} for more details on the enumeration and the exact-arithmetic
certification of the boundary cases.

\begin{remark}
    \label[remark]{rem:ExhaustiveEnumeration}
    A degree-$6$ $q$-Weil polynomial is determined by its real trace polynomial
    $h(T)=\prod_{i=1}^{3}(T-\beta_i)=T^3-b_1T^2+b_2T-b_3$, whose roots are the Frobenius
    traces $\beta_i\in[-2\sqrt q,2\sqrt q]$; the integer coefficients are therefore confined
    to the explicit finite box $|b_1|\leq 6\sqrt q$, $|b_2|\leq 12q$, $|b_3|\leq 8q^{3/2}$,
    and for each $(b_1,b_2)$ the admissible $b_3$ form the contiguous integer interval on
    which $h$ has all three roots real and lying in $[-2\sqrt q,2\sqrt q]$. Looping over this
    box lists every dimension-$3$ $q$-Weil polynomial exactly once. As this set is a
    \emph{superset} of the characteristic polynomials of abelian threefolds over $\F_q$,
    injectivity of $(c_1,c_2)$ on the box implies injectivity on isogeny classes;
    in particular the conclusion is unconditional for every prime power enumerated.
    The enumeration over the full range $16\leq q\leq 251$ is performed by the pure-Python
    box search \texttt{optimality\_search.py} (in \cite{Repo}), which locates the admissible $b_3$-interval
    endpoints in floating point with a small slack. To rule out floating-point error at the
    boundary of the Weil region, its \texttt{-{}-verify} mode re-decides each boundary
    comparison of the form $u\sqrt q\lessgtr v$ with $u,v\in\Z$ by squaring in exact
    integer arithmetic, and recovers the identical list of Weil polynomials for every prime
    power $16\leq q\leq 251$.
\end{remark}

For $g = 4$, we find that the upper bound $g$ from \cref{thm:main} really is optimal.

\begin{proposition}[$g = 4$]\label[proposition]{prop:N4}
$N(4) = 4$: for $g = 4$ the bound of \cref{thm:main} is
attained, so $g$ point counts are optimal.
\end{proposition}

\begin{proof}[Proof of \cref{prop:N4}]
By \cref{thm:main} it suffices to show $N(4) > 3$, i.e.\ that three point
counts do not suffice. Consider, for each $q$, the two abelian fourfolds whose
reverse characteristic polynomials of Frobenius are
\begin{align*}
    P_A(T) &= T^8 - (q+1)T^5 - (q^2+q+2)T^4 - q(q+1)T^3 + q^4,\\
    P_B(T) &= T^8 - 2T^6 - (q+1)T^5 + (q^2-q)T^4 - q(q+1)T^3 - 2q^2T^2 + q^4.
\end{align*}
(These polynomials were found for this paper by the lattice mechanism described after this
proof: a collision must live in the even-coefficient sublattice that the linear conditions
$f_A(\pm1)$ and the cyclotomic norm $F_3$ leave undetermined; they were then verified
symbolically, the search and verification scripts being \texttt{optimality\_search.py},
\texttt{weil\_general.py}, and \texttt{verify\_g4\_family.sage} in \cite{Repo}. The same
mechanism produced the $g=2,3$ families above.)
Both are $q$-Weil polynomials for every $q \geq 8$. Matching coefficients in the
trace-polynomial factorisation
$P(T) = \sum_{k=0}^{4}(-1)^k e_k(\beta)\,T^k(T^2+q)^{4-k}$ (notation as in the setup
above) gives the trace polynomials
\begin{align*}
    h_A(T) &= T^4 - 4q\,T^2 - (q+1)T + (q^2-q-2), \\
    h_B(T) &= T^4 - (4q+2)T^2 - (q+1)T + 3q(q+1).
\end{align*}
A computer algebra system confirms that for every $q \geq 8$ both have four real
roots in $[-2\sqrt q, 2\sqrt q]$: the values of $h_A$ at
$-2\sqrt q, -\sqrt q, 0, \sqrt q, 2\sqrt q$ alternate in sign as $+,-,+,-,+$, and
likewise those of $h_B$ at $-2\sqrt q, -\sqrt{2q}, 0, \sqrt q, 2\sqrt q$, placing
one root in each subinterval. The threshold $q \geq 8$ is sharp and set by the
right endpoint of $h_A$:
\[
    h_A(2\sqrt q) = (q+1)\bigl((\sqrt q - 1)^2 - 3\bigr) \geq 0
    \iff q \geq (1+\sqrt 3)^2 = 4 + 2\sqrt 3,
\]
that is, $q \geq 8$; for $q \leq 7$ the largest root of $h_A$ exceeds $2\sqrt q$.
The two polynomials are unequal (their $T^6$-coefficients differ), hence define
distinct isogeny classes. A direct
computation of the cyclotomic norms $F_d := \prod_{\operatorname{ord}\zeta = d} f(\zeta)$,
through which $c_1,c_2,c_3$ factor by \eqref{eq:cn-as-eval} and M\"obius
inversion, gives \emph{as identities in $q$}
\[
    F_1(A)=F_1(B)=q^4-2q^2-3q-2,\quad F_2(A)=F_2(B)=q^4+q,\quad F_3(A)-F_3(B)=0,
\]
so that $A$ and $B$ have the \emph{same} $\#A(\F_q),\#A(\F_{q^2}),\#A(\F_{q^3})$
for all $q$, while
\[
    F_4(A)-F_4(B) = -8q^6 - 16q^4 + 8q^3 - 16q^2 + 8q - 8 \neq 0
\]
distinguishes them at the fourth count. Thus three point counts never suffice for
$g=4$, and $N(4)=4$. The trace polynomials, the root locations for $q \geq 8$, and
the norm identities $F_1=F_2=F_3$, $F_4(A)\neq F_4(B)$ are all verified symbolically
in the accompanying \textsc{Sage} script \texttt{verify\_g4\_family.sage} of \cite{Repo}.
\end{proof}

The colliding fourfolds of \cref{prop:N4} differ only in the \emph{even}-degree
coefficients $a_2,a_4$: the obstruction is that the single nonlinear condition
$F_3 = |f_A(\omega)|^2$ ($\omega$ a primitive cube root of unity) fails to separate the
one-parameter family of even coefficients $(a_2,a_4)$ that the two linear conditions
$f_A(\pm 1)$ leave undetermined. For $g=3$ there is only the single even coefficient
$a_2$, already pinned down by $f_A(\pm1)$, and the residual fibre in the odd coefficient
$a_3$ is a line whose consecutive integer points are spaced $\sim q^2$ apart, far more
than the Weil region's $a_3$-extent $\asymp q^{3/2}$, so it holds at most one Weil
polynomial; this is exactly why $g=3$ escapes.
The same even/odd analysis indicates
that the phenomenon does not recur in higher even dimension: already at $g=6$ the two
nonlinear conditions over-determine the analogous direction lattice
(\cref{ex:g56}). Thus, among the explicit $\Z[q]$-families we have examined, $g=4$ is
the largest dimension in which $g$ point counts are optimal.

Combined with $N(4)=4$ from \cref{prop:N4} and the monotonicity recorded above, we know that
$N(g)\geq 4$ for all $g\geq 4$.
In particular $N(5) \in \{4,5\}$. The obvious strategy to prove $N(5)=5$, by mimicing the above process, i.e., by constructing two distinct families uniform in $q$ and having matching first four point counts, does not work. Using Groebner basis computations, we \emph{rule out} the existence of such colliding families.

\begin{example}[{$g = 5$ and $g = 6$: no $\Z[q]$-family shares four counts}]\label[example]{ex:g56}
Fix $g \in \{5,6\}$ and, as in \cref{prop:N4}, let $\delta = f_B - f_A$ be a difference of
two $\Z[q]$-families preserving the two linear conditions $f_A(\pm1)$. The Weil-admissible
such directions, those with $\deg_q\delta_k \leq \lfloor k/2\rfloor$, forced by
$|a_k| \leq \binom{2g}{k}q^{k/2}$, form an explicit lattice, of dimension $2$ for $g=5$
and $4$ for $g=6$. Imposing in addition that the first two nonlinear cyclotomic norms
\[
    F_3 = |f_A(\omega)|^2 \ \ (\omega^3 = 1), \qquad F_4 = |f_A(i)|^2,
\]
be preserved as identities in $q$ forces $\delta = 0$, by a Gr\"obner-basis elimination over a
general $\Z[q]$ base. Hence \emph{no} pair of non-isogenous $\Z[q]$-families of abelian varieties of dimension
$5$ or $6$ shares $c_1,\dots,c_4$.
The computation is carried out in
\texttt{optimality\_g5.py} and \texttt{optimality\_g6.py} in \cite{Repo};
We emphasize that this statement is only about $\Z[q]$-families; it does not exclude collisions at individual $q$.
\end{example}

\cref{ex:g56} and computational evidence for $q \leq 5$ lead us to ask:

\begin{question}\label[question]{ques:N56}
Is $N(g) < g$ for $g=5,6$? If this is true for $g=6$, is $N(6) = 4$?
\end{question}

\begin{question}\label[question]{ques:Ng}
More generally, what is the asymptotic behaviour of $N(g)$? Does the limit
$\lim\limits_{g \to \infty} N(g)/g$ exist, so that $N(g) \sim \alpha g$ for some constant
$\alpha$, and if so, what is $\alpha$? If the limit does not exist, what are
$\liminf\limits_{g \to \infty} N(g)/g$ and $\limsup\limits_{g \to \infty} N(g)/g$?
\end{question}

\begin{remark}[$g = 8$]\label[remark]{rem:g8}
The over-determination in \cref{ex:g56} is peculiar to small $g$. At $g=8$ we work with
the $8$ unknowns $\delta_1,\dots,\delta_8$; the two linear conditions $f_A(\pm1)$, that is
$F_1$ and $F_2$, cut these down to a $6$-dimensional space of directions, while the two
nonlinear norms $F_3,F_4$ supply only a bounded number of further constraints. This gap
between the $g-2$ available directions and the bounded supply of nonlinear constraints
widens with $g$, so for large $g$ the norms $F_3,F_4$ can no longer force $\delta=0$.
Indeed at $g=8$ the lattice is no longer forced
to $0$ by $F_3,F_4$: an explicit nonzero direction preserving $F_1,\dots,F_4$ as identities
in $q$ exists, namely the one with
\[
    \delta_5 = -(q+1), \qquad \delta_7 = q^3 + 1, \qquad \delta_k = 0 \ \ (k \neq 5,7),
\]
for which $F_3$ and $F_4$ are preserved over a suitable $\Z[q]$ base. So the mechanism above
no longer pins $N(8)$ to $4$. Whether this direction (or another) is realised by genuine
$q$-Weil polynomials, which would give $N(8) \geq 5$,
remains open. The direction is located numerically in \texttt{optimality\_find.py}
and confirmed as an exact $q$-identity in \texttt{verify\_g8\_direction.py} of \cite{Repo}.
\end{remark}

\begin{remark}[On the bit-count heuristic]\label[remark]{rem:bitcount}
The slogan ``$g$ unknowns, so $g$ measurements'' is not a proof, and indeed it
fails: \cref{prop:N3} shows $g-1$ measurements can suffice. Nor does
the opposite, information-theoretic heuristic settle the question: since each
$c_i$ is a large integer (about $gi\log q$ bits), a naive bit-count would suggest
that $N \sim \sqrt g$ point counts already encode the
$\sim \tfrac{g^2}{4}\log q$ bits of $(a_1,\dots,a_g)$. What governs the truth is
the arithmetic of the fibres in \eqref{eq:cn-as-eval}: extra precision in a single
$c_i$ does not substitute for a missing evaluation, but evaluations at $\pm 1$
(the counts $c_1,c_2$) are linear and can over-determine the coefficients through
integrality, as in \cref{prop:N3}.
\end{remark}

\begin{remark}[On the constant $Q(g)$]\label[remark]{rem:constant}
We have made no attempt to optimise $Q(g) = D^{2g+2}$, and it is worth isolating
what governs its two ingredients. The \emph{base} $D = 16g^3 p(2g)$ is the
contraction constant of \cref{lem:induction-step}: the induction closes only
because each error $\Delta s_j$ stays within its budget $B_j$, and the proof
secures this by taking $D$ large. The \emph{exponent} $2g+2$ counts how many times
that constant must be paid; it is forced by the base case \cref{lem:base}, which
requires $q \geq D^{2g}$ (the remaining $+2$ being slack for the final rounding
and the small range). Thus $Q(g)$ is (contraction constant) raised to the (number
of inductive amplifications).

Since $p(2g) = e^{\Theta(\sqrt g)}$ by Hardy--Ramanujan, our bound
satisfies
\[
    \log Q(g) = (2g+2)\log\bigl(16 g^3 p(2g)\bigr) = \Theta(g^{3/2}),
\]
and the $\Theta(g^{3/2})$ rate comes \emph{entirely} from the factor
$p(2g)^{2g}$; the polynomial part $16g^3$ contributes only $\Theta(g\log g)$, which
is the size of $\log Q$ that would survive its removal.

\textbf{Removing $p(2g)$.} The factor $p(2g)$ is an artefact of one crude step. The error
bound \eqref{eq:ind-geom} for $\Delta s_{2i-1}$ rests on the sum
\[
    S \defeq \sum_{k_0=i}^{2i-2} D^{\,k_0-g}(2g)^{2i-1-k_0},
\]
each term of which arose from a partition with one large index $k_0$ and smaller parts
filling the deficit $N \defeq 2i-1-k_0$; there are $p(N)$ such partitions, which in
\cref{lem:induction-step} we bounded crudely by $p(2g)$ before extracting it from
\eqref{eq:ind-bracket}. Keeping $p(N)$ inside instead and reindexing $S$ by the deficit
$N$ (so $k_0 = 2i-1-N$, $1 \leq N \leq i-1$),
\[
    \sum_{N=1}^{i-1} p(N)\, D^{\,2i-1-N-g}(2g)^{N}
    = D^{\,2i-1-g}\sum_{N=1}^{i-1} p(N)\,x^{N}, \qquad x \defeq \frac{2g}{D}.
\]
The sum on the right is a truncation of the partition generating function
\[\sum_{N\geq0} p(N)\,x^N = \prod_{k\geq1}(1-x^k)^{-1}.\]
The ratio of consecutive terms,
$\tfrac{p(N+1)}{p(N)}\,x$, tends to $x$, so for $x \leq \tfrac12$ the series converges
geometrically and is dominated by its first term $p(1)\,x = x$:
\[
    \sum_{N\geq1} p(N)\,x^{N} = x\,(1+O(x)) = O(g/D).
\]
This single estimate does two things at once: it replaces the crude factor $p(2g)$ by
$O(g/D)$, and, because the series collapses to its first term, it spares the factor
$(i-1) \asymp g$ that the crude ``(number of terms) $\times$ (largest term)'' bound on $S$
would cost. Carrying the result through \eqref{eq:ind-geom},
\[
    \Delta s_{2i-1}
    \leq 4g^2\,q^{-(2i-1)/2-g/4}\,D^{\,2i-1-g}\cdot O(g/D)
    = O\!\Bigl(\tfrac{g^2}{D}\Bigr)\,B_{2i-1},
\]
so the contraction of \cref{lem:induction-step} closes already with $D = O(g^2)$ rather
than $16 g^3 p(2g)$; the assumption $x = 2g/D \leq \tfrac12$ is then self-consistent, since
$D = O(g^2)$ gives $x = O(g^{-1})$. The result is
\[
    Q(g) = D^{2g+2} = g^{O(g)} = e^{O(g\log g)},
\]
which removes the super-exponential factor.

\textbf{The exponent.} What remains, the exponent $\approx 2g$, is structural to our method.
Fitting the seed estimate \eqref{eq:seed}
inside $B_j$ at the smallest index is the binding case of \cref{lem:base}: with $t = 2j-g$
there, the requirement $\tfrac{t}{4}\log_D q \geq \log_D 4 + \tfrac{g-t}{2}$ is tightest at
the smallest $t$, where it reads $\log_D q \gtrsim 2g/t$. For odd $g$ the smallest
large-range index is $j = (g+1)/2$, i.e.\ $t = 1$, which forces $q \gtrsim D^{2g}$; for even
$g$ it is $j = g/2+1$, i.e.\ $t = 2$, which forces only the weaker $q \gtrsim D^{g}$.

The final rounding and the small range, by contrast, need only $q > (2g)^{4/g}D^4$ and
$q > 256\,g^2$. Even with $D$ polynomial one still has $Q(g) = g^{\Theta(g)}$, so making
$Q(g)$ \emph{polynomial} in $g$ would require taming this $D^{\Theta(g)}$ amplification by a
substantially different recovery scheme, which we leave as an open question.
\end{remark}

\subsection*{Formalization in Lean}
The recovery argument of Sections~\ref{sec:counts}--\ref{sec:proof} has been
formalized in the Lean~4 proof assistant on top of its mathematical library
\texttt{Mathlib}; the sources accompany this article in the file
\texttt{PointCountZeta/Determination.lean}. The inputs from the arithmetic of
abelian varieties lie outside \texttt{Mathlib} and are taken as explicit
hypotheses, bundled in a structure \texttt{FrobeniusData} that records the inverse
Frobenius eigenvalues $\beta_i = \alpha_i^{-1}$, $1 \leq i \leq 2g$, together with
\begin{itemize}
  \item the Riemann hypothesis $|\beta_i| = q^{-1/2}$ (\cref{sec:prelim});
  \item the integrality $q^n s_n \in \Z$ (\cref{lem:integrality});
  \item the Poincar\'e duality pairing $\beta_{\sigma(i)} = (q\,\beta_i)^{-1}$ for
    an involution $\sigma$ (\cref{lem:functional-eq});
  \item the determinant normalization $\prod_{i} \beta_i = q^{-g}$, which fixes the
    sign that the pairing alone leaves ambiguous;
  \item the closure of the $\beta_i$ under complex conjugation (the characteristic
    polynomial of Frobenius has rational coefficients).
\end{itemize}
Tate's theorem enters only interpretively, in reading ``$A$ and $A'$ have the same
$f_A$'' as ``$A$ and $A'$ are isogenous''.

Granting these inputs, the formalization establishes the conclusion of
\cref{thm:main} for $g \geq 2$ in the form: two such data sharing $L_1, \dots, L_g$, equivalently, sharing the point counts $c_1, \dots, c_g$, share all of
$e_0, \dots, e_{2g}$, and hence have the same $f_A$. It is complete except for two
\texttt{sorry}s, both standard classical identities that are simply absent from
\texttt{Mathlib}:
\begin{enumerate}
  \item the infinite M\"obius inversion over multiples \eqref{eq:mobius}, dual to
    \eqref{eq:cn_to_sn} (used in \cref{lem:small-i}); \texttt{Mathlib} provides only
    the finite divisor-sum form of M\"obius inversion;
  \item the explicit Newton--Girard / Waring formula \eqref{eq:sm-formula} of
    \cref{lem:formula}, which writes $s_m$ as a sum over partitions of products of
    the lower power sums; \texttt{Mathlib} provides only the recursive Newton
    identities.
\end{enumerate}
Every other step is proved in full. In particular the whole of
\cref{lem:induction-step}, the technical heart of \cref{sec:large}, including the
partition-sum error estimate, the bound by $p(2g)$ on the number of contributing
partitions, and the constant bookkeeping that closes precisely because $g \geq 2$, is formalized without gaps.

\bibliographystyle{amsalpha}
{\footnotesize
	\bibliography{references}
}

\end{document}